\documentclass[12pt, reqno]{amsart}

\newtheorem{theorem}{Theorem}[section]

\newtheorem{corollary}[theorem]{Corollary}   
\newtheorem{definition}[theorem]{Definition}
\newtheorem{example}[theorem]{Example}

\newtheorem{remark}[theorem]{Remark}

\numberwithin{equation}{section}

\usepackage{times}
\usepackage{enumerate}
\usepackage{mathrsfs}
\usepackage{tfrupee}
\usepackage{tikz}
\usetikzlibrary{chains,fit}
\usetikzlibrary{shapes,snakes}
\usetikzlibrary{graphs}
\usepackage{graphicx,adjustbox}
\usepackage{hyperref}
\usepackage{amsmath,amssymb}
\usepackage{amscd}
\usepackage{graphicx}
\usepackage[all]{xy}

\usepackage{float}
\usepackage{caption}
\usepackage{subcaption}

\usepackage{cleveref}

\usepackage{geometry}
 \usepackage[backend=bibtex]{biblatex}
\begin{document}

\title{Cohen-Macaulay weighted oriented chordal and simplicial graphs}
\author{
Kamalesh Saha
}
\date{}

\address{\small \rm Chennai Mathematical Institute, Siruseri, Chennai, Tamil Nadu 603103, India}
\email{ksaha@cmi.ac.in}



\date{}



\allowdisplaybreaks

\begin{abstract}
Herzog, Hibi, and Zheng classified the Cohen-Macaulay edge ideals of chordal graphs. In this paper, we classify Cohen-Macaulay edge ideals of (vertex) weighted oriented chordal and simplicial graphs, a more general class of monomial ideals. In particular, we show that the Cohen-Macaulay property of these ideals is equivalent to the unmixed one and hence, independent of the underlying field.
\end{abstract}

\maketitle

\section{Introduction}

The study of combinatorial commutative algebra began with the pioneering work of R. P. Stanley \cite{stanley75} and G. Reisner \cite{reisner}
in 1975. The study of square-free monomial ideals grabbed the attention of researchers when it was seen in terms of simplicial complexes. Among the sub-classes of square-free monomial ideals, edge ideals of simple graphs, introduced by Villarreal in \cite{vil90}, stand out with their own identity.
\par

The importance of the Cohen-Macaulay property in combinatorics comes through the proof of ``upper bound conjecture". Since Cohen-Macaulay simplicial complexes play an essential role in the study of algebraic geometry, algebraic topology, and combinatorics, finding good classes of Cohen-Macaulay monomial ideals is always in high demand. Because, via the polarization technique, corresponding to any Cohen-Macaulay monomial ideal, we get a Cohen-Macaulay simplicial complex. Corresponding to a simplicial complex $\Delta$, we can associate a graph $G$ such that $I(G)$ is Cohen-Macaulay if and only if $\Delta$ is Cohen-Macaulay. Therefore, the problem of classifying Cohen-Macaulay edge ideals of graphs is as difficult as classifying all Cohen-Macaulay simplicial complexes. There is evidence of edge ideals whose Cohen-Macaulay property depends on the base field. So, one can not expect the general classification, and people are interested in finding those classes of graphs whose Cohen-Macaulay property is independent of the underlying field. In this direction, some remarkable works are as follows: the classification of all Cohen-Macaulay trees by Villarreal \cite{vil01}; all Cohen-Macaulay bipartite graphs by Herzog \& Hibi \cite{hh05}; all Cohen-Macaulay chordal graphs by Herzog, Hibi \& Zheng \cite{cmchordal}.  
\par

In the recent past, the notion of edge ideals of vertex-weighted directed graphs has been defined in \cite{prt19}, which is a generalization of the edge ideals of simple graphs.
\par

A \textit{weighted oriented graph} $D_{G}$, with an underlying simple graph $G$, is a directed graph on the vertex set $V(D_G):=V(G)$ with a weight function $w:V(D_{G})\longrightarrow \mathbb{N}$, where $\mathbb{N}$ denotes the set of positive integers. An edge of $D_{G}$ is denoted by the ordered pair $(x_{i},x_{j})$, which means the direction of the edge is from $x_{i}$ to $x_{j}$.

\begin{definition}{\rm 
Let $D_{G}$ be a weighted oriented graph with $V(D_{G})=\{x_{1},\ldots,x_{n}\}$. Then the \textit{edge ideal} of $D_{G}$, denoted by $I(D_{G})$, is defined as
$$I(D_{G}):=\big<\{x_{i}x_{j}^{w(x_j)}\mid (x_{i},x_{j})\in E(D_{G})\}\big>$$
in the polynomial ring $R=K[x_{1},\ldots,x_{n}]$ over a field $K$. When $w(x_i)=1$ for all $i$, $I(D_G)$ coincide with the usual edge ideal $I(G)$ of $G$. By saying $D_{G}$ or $I(D_{G})$ is Cohen-Macaulay, we mean the quotient ring $R/I(D_{G})$ is Cohen-Macaulay.
}
\end{definition}

The motivation behind studying weighted oriented edge ideals has its foundation in coding theory. Specifically, they appear as the initial ideals of certain vanishing ideals $I(\mathcal{X})$ of some sets of projective points $\mathcal{X}$ in the study of Reed-Muller-type codes (see \cite{cnl17}, \cite{mpv17}). One can estimate some basic parameters and properties of Reed-Muller-type codes associated with $\mathcal{X}$ by studying weighted oriented edge ideals. For example, if $I(D_{G})$ is Cohen-Macaulay, then $I(\mathcal{X})$ is Cohen-Macaulay.\par

Like the classification problem of Cohen-Macaulay simple graphs, people are interested in finding suitable classes of Cohen-Macaulay weighted oriented graphs (independent of the base field). For certain classes of weighted oriented edge ideals, Cohen-Macaulay ones are characterized combinatorially: (i) path and complete graphs \cite{prt19}; (ii) forests \cite{gmsvv18}; (iii) bipartite graphs and graphs with perfect matching \cite{hlmrv19}; (iv) {K}\H{o}nig graphs \cite{prv21}; (v) cycles \cite{ki_weight}. 
\par

Chordal graphs play a significant role in the theory of edge ideals of simple graphs. This class is reflected in the celebrated Fr\"{o}berg's theorem \cite{frob90}, which states that $I(G)$ has a linear resolution if and only if $G^c$ (complement of $G$) is chordal. Again, Francisco and Van Tuyl showed that if $G$ is a chordal graph, then $R/I(G)$ is sequentially Cohen-Macaulay \cite{fv07}. Herzog et al. gave the combinatorial characterization of Cohen-Macaulay chordal graphs in \cite{cmchordal}. For chordal graphs, they proved that the Cohen-Macaulay property of $I(G)$ is equivalent to the unmixed one and, thus, independent of the field $K$. 
\par

In this article, we characterize all Cohen-Macaulay chordal and simplicial weighted oriented graphs. In \cite{cpr22}, the unmixed chordal and simplicial weighted oriented graphs were classified. We show in Theorem \ref{thmchordalcm} that if the underlying graph $G$ of a weighted oriented graph $D_{G}$ is chordal or simplicial, then $I(D_G)$ is Cohen-Macaulay if and only if $I(D_G)$ is unmixed. This ensures that the Cohen-Macaulay property of edge ideals of weighted oriented chordal and simplicial graphs is independent of the base field. Our result generalizes the theorem of Herzog, Hibi \& Zheng \cite{cmchordal} and includes a larger class of monomial ideals. Also, as a corollary, we identify all Gorenstein edge ideals of weighted oriented chordal and simplicial graphs.
\section{Preliminaries}\label{preli}

In this section, we recall some definitions and results related to our work. By a simple graph, we mean an undirected graph without multiple edges or loops.\par

Let $G$ be a simple graph. A \textit{vertex cover} $C$ of $G$ is a subset of $V(G)$ such that $e\cap C\neq \emptyset$ for all $e\in E(G)$. A \textit{minimal vertex cover} of $G$ is a vertex cover $C$ of $G$ such that if $C'\subsetneq C$, then $C'$ can not be a vertex cover of $G$. Let $D_{G}$ be a weighted oriented graph. For a vertex $v\in V(G)$, we call $\mathcal{N}_{G}(v):=\{u\in V(G)\mid \{u,v\}\in E(G)\}$ the \textit{neighbour set} of $v$ in $G$. We write $\mathcal{N}_{G}[v]:=\mathcal{N}_{G}(v)\cup \{v\}$. A graph $G$ is said to be \textit{complete} if there is an edge between every pair of vertices. For $A\subseteq V(G)$, we denote the induced subgraph of $G$ on the vertex set $A$ by $G[A]$.

\begin{definition}{\rm
A \textit{cycle} of length $n$ is a connected graph on $n$ vertices having exactly $n$ edges. A graph $G$ is called a \textit{chordal graph} if $G$ has no induced cycle of length $>3$.
}
\end{definition}

\begin{definition}{\rm
A vertex $v\in V(G)$ is said to be a \textit{simplicial vertex} if the induced subgraph of $G$ on $\mathcal{N}_{G}[v]$, i.e. $G[\mathcal{N}_G[v]]$ is a complete graph and $G[\mathcal{N}_G[v]]$ is called a \textit{simplex} of $G$. A graph $G$ is said to be a \textit{simplicial graph} if every vertex of $G$ is a simplicial vertex of $G$ or is adjacent to a simplicial vertex of $G$.
}
\end{definition}

\begin{definition}[\cite{prt19}]{\rm
Let $D_{G}$ be a weighted oriented graph. Corresponding to a vertex cover $C$ of $G$, consider the following sets
\begin{align*}
& \mathcal{L}_{1}(C)=\{x\in C \mid \exists\, (x,y)\in E(D_{G})\,\, \text{such that}\,\, y\not\in C\},\\
& \mathcal{L}_{2}(C)=\{x\in C\mid x\not\in\mathcal{L}_{1}(C)\,\,\text{and}\,\, \exists\, (y,x)\in E(D_G)\,\,\text{such that}\,\, y\not\in C\},\\
& \mathcal{L}_{3}(C)=C\setminus (\mathcal{L}_{1}(C)\cup \mathcal{L}_{2}(C))=\{x\in C\mid \mathcal{N}_{G}(x)\subseteq C\}.
\end{align*}
A vertex cover $C$ of $G$ is called a \textit{strong vertex cover} of $D_{G}$ if $C$ is either a minimal vertex cover of $G$ or for all $x\in \mathcal{L}_{3}(C)$, there is an edge $(y, x)\in E(D_{G})$ with $y\in \mathcal{L}_{2}(C)\cup \mathcal{L}_{3}(C)$ and $w(y)\neq 1$.
}
\end{definition}

\begin{remark}\label{remmvc}{\rm
Let $D_G$ be a weighted oriented graph. Then radical of $I(D_G)$ is $I(G)$. Thus, the minimal primes of $I(D_G)$ are the associated primes of $I(G)$, and they are precisely the ideals generated by the minimal vertex covers of $G$.
}
\end{remark}

For a vertex cover $C$ of $D_G$, consider the following irreducible ideal associated to $C$ 
$$Q_{C}:=\big< \mathcal{L}_{1}(C)\cup\{x_{j}^{w(x_j)}\mid x_{j}\in \mathcal{L}_{2}(C)\cup \mathcal{L}_{3}(C)\}\big>.$$

\begin{theorem}[{\cite[Theorem 25]{prt19}}]\label{prmdc}
Let $\mathcal{C}_{s}$ denote the set of all strong vertex covers of $D_{G}$. Then the irredundant irreducible primary decomposition of 
$I(D_{G})$ is given by 
$$I(D_{G}) = \bigcap_{C\in\mathcal{C}_{s}} Q_{C}.$$
Moreover, $\{P_C\mid P_C=\big<C\big>,\,\,\text{where}\,\, C\in\mathcal{C}_{s}\}$ is the set of associated primes of $I(D_G)$.
\end{theorem}

\begin{theorem}[{\cite[Theorem 31]{prt19}}]\label{unm}
$I(D_{G})$ is unmixed if and only if $I(G)$ is unmixed and $\mathcal{L}_{3}(C)=\emptyset$ for any strong vertex cover $C$ of $D_{G}$.
\end{theorem}
\medskip

\section{The Cohen-Macaulay Classification}\label{secclassification}

In this section, we will show that the Cohen-Macaulay property of weighted oriented chordal and simplicial graphs is equivalent to the unmixed property. For these classes of graphs, we also characterize the Gorenstein ones.

\begin{theorem}\label{thmchordalcm}
Let $D_{G}$ be a weighted oriented graph such that $G$ is chordal or simplicial. Then $I(D_G)$ is Cohen-Macaulay if and only if $I(D_G)$ is unmixed.
\end{theorem}

\begin{proof}
The ``only if" part of the theorem is well-known. Let $I(D_G)$ be unmixed. Then $I(G)$ is unmixed. Therefore, by \cite[Theorem 1 and 2]{ptv96}, every vertex of $G$ belongs to exactly one simplex of $G$, and the vertex sets of simplices form a partition of $V(G)$. Let $H_{1},\ldots,H_{m}$ be all the simplices of $D_G$. Without loss of generality, let $V(H_{1})=\{x_{1},\ldots,x_{k_1}\}$, $V(H_{2})=\{x_{k_1+1},\ldots, x_{k_2}\},\ldots, V(H_m)=\{x_{k_{m-1}+1},\ldots, x_{k_m}\}$ and $x_{k_i}$ is a simplicial vertex of the simplex $H_{i}$ for each $1\leq i\leq m$. Let $h_{i}=x_{k_{i-1}+1}+\cdots+x_{k_i}$ for each $1\leq i\leq m$, where $k_0=0$.\medskip

\noindent\textbf{Claim:} The polynomials $h_{1},\ldots,h_{m}$ forms a regular sequence on $R/I(D_{G})$.
\medskip

\noindent\textit{Proof of the claim.} $I(D_{G})$ is a monomial ideal, and so, every associated prime of $I(D_G)$ is generated by a set of 
variables. Since $I(D_G)$ is unmixed, all the associated primes of $I(D_G)$ are minimal, and they are the associated primes of the ideal $I(G)$, the radical of $I(D_G)$. Thus, by Remark \ref{remmvc}, the associated primes of $I(D_G)$ are generated by the minimal vertex covers of $G$. Note that no minimal vertex cover of $G$ contain the set $V(H_{1})$ as $\mathcal{N}_{G}[x_{k_{1}}]=V(H_{1})$. Therefore, $h_{1}$ does not belong to any associated prime of $I(D_G)$. Hence, $h_1$ is a regular element on $R/I(D_G)$. Now, we will show that $h_{i}$ is regular on $R/\big<I(D_G), h_1,\ldots,h_{i-1}\big>$. Let $I(D_G)=\bigcap_{C\in\mathcal{C}_{s}} Q_{C}$ be the primary decomposition of $I(D_G)$ as described in Theorem \ref{prmdc}. Let $P$ be an associated prime of the ideal $\big<I(D_{G}),h_1,\ldots,h_{i-1}\big>$. Then $I(D_G)\subseteq P$ implies $Q_{C}\subseteq P$ for some $C\in \mathcal{C}_{s}$ and hence, $P_{C}\subseteq P$. Since $H_{j}$ is a simplex for each $1\leq j\leq m$, at least $\vert V(H_{j})\vert-1$ elements of $V(H_{j})$ belong to $P$. Therefore, for each $1\leq j\leq i-1$, we have $V(H_{j})\subseteq P$ as $h_1,\ldots,h_{i-1}\in P$. Now it is clear that $P_{C}+\big<\bigcup_{j=1}^{i-1}V(H_j)\big>\subseteq P$. Suppose $P_{C}+\big<\bigcup_{j=1}^{i-1}V(H_j)\big>\subsetneq P$. Then $P$ has an element $g$ such that $g\not\in P_{C}+\big<\bigcup_{j=1}^{i-1}V(H_j)\big>$. Then we can choose $g$ from the polynomial ring $K[V(G)\setminus C\cup V(H_1)\cup\cdots\cup V(H_{i-1})]$. Since $P$ is an associated prime ideal of $\big<I(D_G),h_1,\ldots,h_{i-1}\big>$ and $h_1,\ldots,h_{i-1}$ are linear forms, we can choose a homogeneous polynomial $f$ such that $\big<I(D_G),h_1,\ldots,h_{i-1}\big>:f=P$ and $g\in I(D_G):f$. Then $I(D_G)$ will have an associated prime containing $\big<P_C,g\big>$, which is a contradiction as $I(D_G)$ being unmixed has no embedded prime. Hence, any associated prime of $\big<I(D_G),h_1,\ldots,h_{i-1}\big>$ is of the form $P_C+\big<\bigcup_{j=1}^{i-1}V(H_j)\big>$ for some $C\in\mathcal{C}_{s}$. Now from this observation, it is clear that the element $h_{i}$ does not belong to $P$, otherwise $V(H_{i})\subseteq P$, which is a contradiction as $V(H_{i})\cap V(H_{j})=\emptyset$ for all $1\leq j\leq i-1$ and $C$, being the minimal cover of $D_G$, can not contain the set $V(H_{i})$. We have chosen $C\in \mathcal{C}_{s}$ in an arbitrary fashion and thus, $h_{i}$ can not belong to any associated prime of $\big<I(D_{G}),h_1,\ldots,h_{i-1}\big>$. Hence, $h_{i}$ is a regular element on $R/\big<I(D_{G}),h_1,\ldots,h_{i-1}\big>$. By induction hypothesis, we get $h_{1},\ldots,h_{m}$ forms a regular sequence on $R/I(D_{G})$. Therefore, $\mathrm{depth}(R/I(D_{G}))\geq m$. Now $I(D_{G})$ is unmixed, which implies any associated prime of $I(D_{G})$ is generated by some minimal vertex cover of $G$. Since $V(H_{1}),\ldots, V(H_{m})$ forms a partition of $V(D_{G})$, each minimal vertex cover of $G$ contain exactly $\sum_{i=1}^{m}(\vert V(H_i)\vert -1)= \vert V(D_{G})\vert-m$ vertices. Thus, $\mathrm{ht}(I(D_{G}))=\vert V(D_{G})\vert -m$ and $\mathrm{dim}(R/I(D_G))=\vert V(D_{G})\vert -\mathrm{ht}(I(D_G))=m$. It is well-known that $\mathrm{depth}(R/I(D_G))\leq \mathrm{dim}(R/I(D_G))$, which gives $\mathrm{depth}(R/I(D_G))= \mathrm{dim}(R/I(D_G))=m$. Hence $R/I(D_G)$ is Cohen-Macaulay.
\end{proof}

\begin{corollary}
Let $D_G$ be a weighted oriented graph such that $G$ is chordal or simplicial. Then $R/I(D_G)$ is Gorenstein if and only if $D_G$ is a disjoint union of edges.
\end{corollary}

\begin{proof}
Let $G$ be a chordal or simplicial graph such that $I(G)$ is unmixed. Then by \cite[Theorem 1 and 2]{ptv96}, each vertex of $G$ belongs to exactly one simplex of $G$, and the vertex sets of simplices form a partition of $V(G)$. Therefore, the same argument given for chordal graphs in \cite[Corollary 2.1]{cmchordal} is also applicable to simplicial graphs. Hence, it follows from \cite[Corollary 2.1]{cmchordal} that $R/I(G)$ is Gorenstein if and only if $G$ is a disjoint union of edges. Now $I(D_G)$ being a monomial ideal, by \cite[Theorem 2.6]{htt05}, $R/I(D_G)$ is Gorenstein implies $R/I(G)$ is Gorenstein. Thus, $D_G$ is a disjoint union of edges.\par 

Conversely, if $D_G$ is a disjoint union of edges, then $I(D_G)$ is complete intersection, and hence, $R/I(D_G)$ is Gorenstein.
\end{proof}

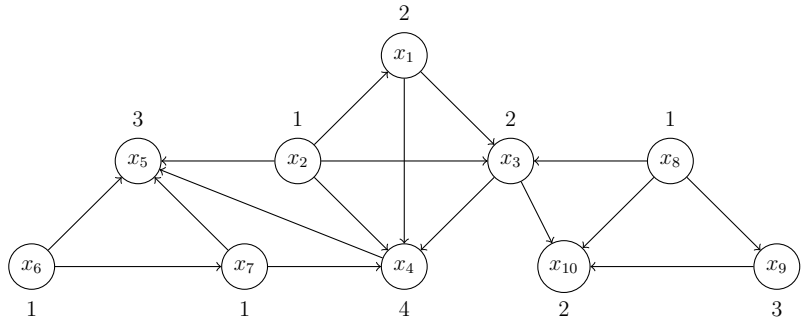
\begin{figure}[H]
\begin{center}
\begin{tikzpicture}
   [scale=1.4,auto=left,every node/.style={circle,scale=0.7}]
 
  \node[draw] (n1) at (0,2)  {$x_{1}$};
  \node[draw] (n2) at (-1,1)  {$x_{2}$};
  \node[draw] (n3) at (1,1) {$x_{3}$};
   \node[draw] (n4) at (0,0) {$x_{4}$};
   \node[draw] (n5) at (-2.5,1) {$x_{5}$};
 \node[draw] (n6) at (-3.5,0) {$x_{6}$};
 \node[draw] (n7) at (-1.5,0) {$x_{7}$};
 \node[draw] (n8) at (2.5,1) {$x_{8}$};
 \node[draw] (n9) at (3.5,0) {$x_{9}$};
 \node[draw] (n10) at (1.5,0) {$x_{10}$};
 \node (m1) at (0,2.4){$2$};
 \node (m2) at (-1,1.4){$1$};
 \node (m3) at (1,1.4) {$2$};
 \node (m4) at (0,-0.4) {$4$};
 \node (m5) at (-2.5,1.4) {$3$};
 \node (m6) at (-3.5,-0.4) {$1$};
 \node (m7) at (-1.5,-0.4) {$1$};
 \node (m8) at (2.5,1.4){$1$};
 \node (m9) at (3.5,-0.4){$3$};
 \node (m10) at (1.5,-0.4){$2$};
 
  \foreach \from/\to in {n2/n1, n1/n3, n1/n4, n2/n3, n2/n4, n3/n4, n6/n5, n7/n5, n6/n7, n8/n9, n9/n10, n8/n10, n2/n5, n8/n3, n7/n4, n4/n5, n3/n10}
    \draw[->] (\from) -- (\to); 
\end{tikzpicture}
\end{center} 
\caption{A Cohen-Macaulay weighted oriented chordal graph $D_G$.}\label{fig1}
\end{figure}

\begin{example}{\rm
Consider the weighted oriented chordal graph $D_G$ in Figure \ref{fig1}. Note that the induced subgraphs $G[\{x_1,x_2,x_3,x_4\}]$, $G[\{x_5,x_6,x_7\}]$ and $G[\{x_8,x_9,x_{10}\}]$ are simplices of $G$ and each vertex of $G$ belongs to exactly one of these simplices. Then by \cite[Theorem 2]{ptv96}, $I(G)$ is unmixed. Let $C$ be a strong vertex cover of $D_G$. If $\{x_1,x_2,x_3,x_4\}\subseteq C$, then $x_1\in\mathcal{L}_{3}(C)$, but, there is no vertex $x_i$ with $w(x_i)>1$ and $(x_i,x_1)\in E(D_G)$. This gives a contradiction to the definition of strong vertex cover. Therefore, $\{x_1,x_2,x_3,x_4\}\not\subseteq C$ and similarly, $\{x_5,x_6,x_7\}, \{x_8,x_9,x_{10}\}\not\subseteq C$. Hence, $\mathcal{L}_{3}(C)=\emptyset$, and this is true for any strong vertex cover of $D_G$. Thus, by Theorem \ref{unm}, $I(D_G)$ is unmixed and from Theorem \ref{thmchordalcm}, we get $R/I(D_G)$ is Cohen-Macaulay.
}
\end{example}

\section*{Acknowledgment}
The author would like to thank the National Board for Higher Mathematics (India) for the financial support through NBHM Postdoctoral Fellowship. Also, the author is thankful to Prof. Tran Nam Trung for his valuable comment.

\printbibliography

\end{document}